\documentclass[12pt,a4,notitlepage]{amsart}
\usepackage{amsfonts}

\newcommand{\w}{\omega}
\newcommand{\e}{\varepsilon}
\newcommand{\IR}{\mathbb R}
\newcommand{\IN}{\mathbb N}

\newcommand{\Ra}{\Rightarrow}

\newtheorem{example}{Example}

\newtheorem{theorem}{Theorem}
\newtheorem{proposition}{Proposition}
\newtheorem{corollary}{Corollary}
\newtheorem{problem}{Problem}

\theoremstyle{definition}

\begin{document}

\title{Metrizable  quotients of $C_p$-spaces}
\author{T. Banakh, J. K{\c{a}}kol  and W. \'Sliwa}
\address{Faculty of Mechanics and Mathematics, Ivan Franko National University of Lviv, Universytetska St. 1, Lviv 79000, Ukraine 2, and  Institute of Mathematics, Jan Kochanowski University in Kielce, ul. Swi\c etokrzyska 15, 25-406 Kielce, Poland}

\email{t.o.banakh@gmail.com}

\address{Faculty of Mathematics and Informatics. A. Mickiewicz University  61-614 Pozna{\'{n}} , and Institute of Mathematics Czech Academy of Sciences, Prague}

\email{kakol@amu.edu.pl}
\address{Faculty of Mathematics and Natural Sciences,
University of Rzesz\'ow%
\newline
\indent%
35-310 Rzesz\'ow, Poland }
\email{sliwa@amu.edu.pl; wsliwa@ur.edu.pl}


\begin{abstract}
The famous  Rosenthal-Lacey theorem asserts that  for each infinite compact set $K$ the  Banach space $C(K)$ admits a quotient which is either a copy of $c$ or $\ell_{2}$. What is the case when the uniform topology of $C(K)$ is replaced  by the pointwise topology? Is it true that $C_p(X)$ always has an infinite-dimensional separable (or better metrizable) quotient? In this paper we prove that for a Tychonoff space $X$ the function space $C_p(X)$ has an infinite-dimensional metrizable quotient if $X$ either contains an infinite discrete $C^*$-embedded subspace or else $X$ has a sequence $(K_n)_{n\in\IN}$ of compact subsets such that for every $n$ the space $K_n$ contains two disjoint topological copies of $K_{n+1}$. Applying the latter result, we show that under $\lozenge$ there exists a zero-dimensional Efimov space $K$ whose function space $C_{p}(K)$ has an infinite-dimensional metrizable quotient. These two theorems essentially improve earlier results of K\c akol and Sliwa on infinite-dimensional separable quotients of $C_p$-spaces.
\end{abstract}
\maketitle

\section{Introduction}
Let $X$ be  a Tychonoff space. By $C_{p}(X)$ and $C_{c}(X)$ we denote the space of real-valued continuous functions on $X$ endowed with the pointwise and the compact-open topology, respectively.

The classic Rosenthal-Lacey theorem, see \cite{Ro}, \cite{JR},  and \cite{La},  asserts that the Banach space $C(K)$  of continuous real-valued maps on an infinite compact space $K$ has a  quotient  isomorphic to $c$ or $\ell_{2}.$

This theorem motivates the following  natural question (first discussed in \cite{kakol-sliwa}):

\begin{problem}[K\c akol, Sliwa]\label{prob1} Does  $C_{p}(K)$ admit an infinite-dimensional separable quotient for any infinite compact space $K$?
\end{problem}

In particular, does  $C_{p}(\beta\mathbb{N})$ admit an infinite-dimensional separable quotient (shortly $SQ$)? Our main theorem of \cite[Theorem 4]{kakol-sliwa}  showed  that $C_{p}(K)$ has $SQ$ for any compact space $K$ containing a copy of $\beta\mathbb{N}$. Consequently, this theorem  reduces  Problem~\ref{prob1} to the case when $K$ is an \emph{Efimov space} (i.e. $K$ is an infinite compact space that contains neither a non-trivial  convergent sequence nor a copy of $\beta\mathbb{N}$). Although, it is unknown if Efimov spaces exist in ZFC (see  \cite{dow}, \cite{dow1}, \cite{dow2}, \cite{efimov}, \cite{fedorczuk1},  \cite{fedorczuk2},  \cite{geschke},  \cite{hart})  we showed in \cite{kakol-sliwa} that under $\lozenge$ for some  Efimov spaces $K$ the function space $C_{p}(K)$ has $SQ$.

 On the other hand, in \cite{sa2} it was shown   that  $C_{p}(K)$   has an \emph{infinite-dimensional separable quotient algebra} if and only if  $K$ contains an infinite countable closed subset. Hence $C_{p}(\beta\mathbb{N})$ lacks infinite-dimensional separable quotient algebras.

Clearly Problem~\ref{prob1} is  motivated by Rosenthal-Lacey theorem, but  one can provide  more specific  motivations. Indeed, although it is unknown whether $C_c(X)$ or $C_p(X)$ always has $SQ$,  some partial results are known:  If $X$ is of pointwise countable type, then $C_c(X)$ has a quotient isomorphic to either $\mathbb{R}^{\mathbb{N}}$ or $c$ or $\ell_{2}$, see \cite[Corollary 22]{sa2}. Also $C_{c}(X)$ has $SQ$ provided $C_{c}(X)$ is barrelled, \cite{sa1}.  Recall that all first countable spaces and all locally compact spaces are of  pointwise countable type, see \cite{engel}.

In \cite[Corollary 11]{sa2} we proved that for a fixed Tychonoff space $X$, if $C_{p}(X)$ has $SQ$, then also $C_{c}(X)$ has $SQ$. Conversely, if $C_{c}(X)$ has $SQ$ and for every infinite compact $K\subset X$ the space $C_{p}(K)$ has $SQ$, then $C_{p}(X)$ has also $SQ$. Indeed,
two cases are possible.
\begin{enumerate}
\item \emph{Every compact subset of $X$ is finite. } Then $C_{p}(X)=C_{c}(X)$ is barrelled and \cite{sa1} applies to get that $C_{p}(X)$ has $SQ$.
\item \emph{$X$ contains an infinite compact subset $K$. } Then $C_{p}(K)$ has $SQ$ (by assumption). Since the restriction map $f\rightarrow f|K$, $f\in C(X)$ is  a continuous open surjection from $C_{p}(X)$ onto $C_{p}(K)$, the desired conclusion  holds.
\end{enumerate}
 Let $X$ be a Tychonoff space. First observe that each metrizable quotient of $C_p(X)$ is separable. Indeed, this follows from the separability of metizable spaces of countable cellularity and the fact that  $C_p(X)$ has countable cellularity, being a dense subspace of $\IR^X$, see \cite{Arhangel}.

Although Problem~\ref{prob1} is still left open, but being motivated by the above mentioned partial  results one can formulate the following much  stronger version of thi question.

\begin{problem}\label{problem}
For which Tychonoff spaces $X$ the function space $C_{p}(X)$ admits a metrizable infinite-dimensional quotient?
\end{problem}

This problem has a simple solution for Tychonoff spaces which are not pseudocompact. We recall that a Tychonoff space $X$ is {\em pseudocompact} if each continuous real-valued function on $X$ is bounded.

\begin{proposition}\label{first} For a Tychonoff space $X$ the following conditions are equivalent:
\begin{enumerate}
\item $X$ is not pseudocompact;
\item $C_p(X)$ has a subspace, isomorphic to $\mathbb R^{\mathbb N}$;
\item $C_p(X)$ has a quotient space, isomorphic to $\mathbb R^{\mathbb N}$;
\item $C_p(X)$ admits a linear continuous map onto $\mathbb R^{\mathbb N}$.
\end{enumerate}
\end{proposition}

\begin{proof} The implication $(1)\Ra(2)$ is proved in Theorem 14 of \cite{sa2} and $(2)\Ra(3)$ follows from the complementability of $\mathbb{R}^{\mathbb{N}}$ in any locally convex space containing it, see \cite[Corollary 2.6.5]{bonet}. The implication $(3)\Ra(4)$ is trivial. To see that $(4)\Ra(1)$, observe that for a pseudocompact space $X$ the function space $C_p(X)$ is $\sigma$-bounded, which means that it can be written as the countable union $C_p(X)=\bigcup_{n=1}^\infty\{f\in C_p(X):\sup_{x\in X}|f(x)|\le n\}$ of bounded subsets. Then the image of $C_p(X)$ under any linear continuous operator also is $\sigma$-bounded. On the other hand, the Baire Theorem ensures that the space $\IR^{\IN}$ is not $\sigma$-bounded.
\end{proof}

The main results of paper are the following two theorems giving two partial answers to Problem \ref{problem}.

\begin{theorem}\label{main1}  If a pseudocompact Tychonoff space $X$ contains an infinite discrete $C^*$-embedded subspace $D$, then the function space $C_p(X)$ has an infinite-dimensional metrizable quotient. More pricesely, for any sequence $(F_n)_{n=1}^\infty$ of non-empty, finite and pairwise disjoint subsets of $D$ with $\lim_n |F_n|=\infty$ and the linear subspace $$Z=\bigcap_{n=1}^{\infty} \{f\in C_p(X): \textstyle{\sum\limits_{x\in F_n} f(x)=0}\}$$ the quotient space $C_p(X)/Z$ is isomorphic  to the subspace $\ell_\infty=\{(x_n)\in \IR^\mathbb{N}:\sup_n |x_n|<\infty\}$ of $\IR^\mathbb{N}$.
\end{theorem}

A subspace $A$ of a topological space $X$ is called {\em $C^*$-embedded} if each bounded continuous function $f:A\to\IR$ has a continuous extension $\bar f:X\to\IR$.

If a Tychonoff space $X$ is compact, then $X$ contains an infinite discrete $C^*$-embedded subspace if and only if $X$ contains a copy of $\beta\mathbb{N}$. On the other hand, the space $\omega_{1}$ is pseudocompact noncompact which does not contain $C^*$-embedded infinite discrete subspaces. Moreover, the space $\Lambda:=\beta\mathbb{R}\setminus(\beta\mathbb{N}\setminus\mathbb{N})$ discussed in \cite[6P, p.97]{gilman} is pseudocompact, noncompact and contains $\mathbb{N}$ as a closed discrete $C^{*}$-embedded set.


\begin{corollary}\label{c1} For any infinite discrete space $D$ the  space $C_p(\beta D)$ has a quotient space, isomorphic to the subspace $\ell_\infty$ of $\IR^\mathbb{N}.$
\end{corollary}
Theorem  \ref{main1} and Proposition \ref{first} yield immediately
\begin{corollary}
For any Tychonoff space $X$ containing a $C^*$-embedded infinite discrete subspace, the function space $C_p(X)$ has an infinite-dimensional metrizable quotient, isomorphic to $\IR^{\IN}$ or $\ell_\infty$.
\end{corollary}

Besides the subspace $\ell_\infty$ of $\IR^{\IN}$, the following corollary of Theorem~\ref{main1} involves also the subspace $c_0:=\{(x_n)_{n\in\IN}\in\IR^\IN:\lim_{n\to\infty}x_n=0\}$ of $\IR^\IN$.

\begin{corollary}\label{corollary 2} If a compact Hausdorff space $X$ is not Efimov, then its function space $C_{p}(X)$ has a quotient space, isomorphic to the subspaces $\ell_\infty$ or $c_0$ in $\IR^\mathbb{N}$.
\end{corollary}

\begin{proof} The space $X$, being non-Efimov, contains either an infinite converging sequence or a copy of $\beta\mathbb{N}$.  In the latter case $X$ contains an infinite discrete $C^{*}$-embedded subspace and Theorem \ref{main1} implies that $C_p(X)$ has a quotient space, isomorphic to $\ell_\infty\subset\IR^\IN$. If $X$ contains a sequence    $(x_n)_{n\in\IN}$ of pairwise distinct points that converges to a point $x\in X$, then for the compact subset $K:=\{x\}\cup\{x_n\}_{n\in\IN}$ of $X$ the function space $C_{p}(K)$ is isomorphic to $c_0\subset\IR^\IN$ and is complemented in $C_{p}(X)$, see \cite[Theorem 1, p.130, Proposition 2, p.128]{arkh-encyklo}.
\end{proof}

We prove also the following theorem which will be applied for Example \ref{exa} below.
\begin{theorem}\label{main2} For a Tychonoff space $X$ the space $C_p(X)$ has a metrizable infinite-dimensional quotient if there exists a sequence $(K_n)_{n\in\w}$ of non-empty compact subsets of $X$ such that for every $n\in\w$ the compact set $K_n$ contains two disjoint topological copies of $K_{n+1}$.
\end{theorem}

We do not know if in Theorem \ref{main2} the obtained quotient is isomorphic to $\ell_{\infty}$ or $c_0$.
 Example \ref{exa} below provides an Efimov space $K$ (under $\lozenge$) for which Theorem \ref{main2} applies.
  \begin{example}\label{exa} Under $\lozenge$ there exists a Efimov space $K$ whose function space $C_{p}(K)$ has a metrizable infinite-dimensional quotient.
  \end{example}

 \begin{proof}
   De la Vega \cite[Theorem 3.22]{vega1} (we refer also to \cite{vega})  constructed under $\lozenge$ a compact zero-dimensional hereditary separable  space $K$ (hence not containing a copy of $\beta\mathbb{N}$) such that:
  \begin{enumerate}
  \item [(i)] $K$ does not contain non-trivial convergent sequences.
   \item [(ii)]  $K$  has a  base of clopen pairwise hemeomorphic sets.
    \end{enumerate}
    It is easy to see that $K$  admits  a sequence $(K_{n})$ of infinite compact subsets such that each $K_{n}$ contains two disjoint subsets homeomorphic to $K_{n+1}$; therefore by Theorem \ref{main2} the space $C_{p}(K)$ has the desired property.
  \end{proof}

  As the space $K$ from Example~\ref{exa} does  not contain $\beta\mathbb{N}$, the assumption of Theorem \ref{main1} is not satisfied. Note that in \cite[Example 17]{kakol-sliwa} we provided  (again under  $\lozenge$) an example  of an Efimov space $K$ for which Theorem \ref{main2} cannot be applied.

\section{ Proof of Theorem \ref{main1}}

Let $D$ be an infinite discrete $C^*$-embedded subspace of a pseudocompact Tychonoff space $X$. Choose any sequence $(F_n)_{n\in\IN}$ of non-empty, finite and pairwise disjoint subsets of $D$ with $\lim_{n\to\infty} |F_n|= \infty$.

For every $n\in\mathbb{N}$ consider the probability measure $\mu_n=\frac1{|F_n|}\sum_{x\in F_n}\delta_x$, where $\delta_x$ is the Dirac measure concentrated at $x$.

The pseudocompactness of the space $X$ guarantees that the linear continuous operator $$T:C_p(X)\to\ell_\infty\subset\IR^{\IN},\;\;T:f\mapsto(\mu_n(f))_{n\in\IN}$$is well-defined.

We claim that the operator $T$ is open. Given a neighborhood $U\subset C_p(X)$ of zero, we need to check that $T(U)$ is a neighborhood of zero in $\ell_\infty$.
We can assume that $U$ is of the basic form $$U:=\{f\in C_p(X):\max_{x\in E}|f(x)|<\e\}$$ for some finite set $E\subset X$ and some $\e>0$.

Choose a number $m\in\IN$ such that $\inf_{k>m} |F_k|\ge 2(|E|+1)$.
We claim that $T(U)$ contains the open neighborhood $$V:=\{(y_k)_{k=1}^\infty\in\ell_\infty:\max_{k\le m}|x_k|<\e\}$$ of zero in $\ell_\infty\subset\IR^{\IN}$.

Fix any sequence $(y_k)_{k=1}^\infty\in V$. Choose any partition $\mathcal P$ of the set $D\setminus\bigcup_{k=1}^mF_k$ into $|\mathcal P|=|E|+1$ pairwise disjoint sets such that for every $P\in\mathcal P$ and $k>m$ the intersection $P\cap F_k$ has cardinality $$|P\cap F_k|\ge \frac{|F_k|}{|E|+1}-1\ge 1.$$ Since the discrete subspace $D$ is $C^*$-embedded in $X$, the sets in the partition $\mathcal P$ have pairwise disjoint closures in $X$. Taking into account that $|\mathcal P|>|E|$, we can find a set $P\in\mathcal P$ whose closure $\bar P$ is disjoint with $E$.

Consider the function $f:S\to\IR$ defined by
$$f(x)=\begin{cases}y_k&\mbox{if $x\in F_k$ for some $k\le m$};\\
y_k{\cdot}\frac{|F_k|}{|F_k\cap P|}&\mbox{if $x\in P\cap F_k$ for some $k>m$};\\
0&\mbox{otherwise}.
\end{cases}
$$
The function $f$ is bounded because $\sup_{k\in\IN}|y_k|<\infty$ and
$$\sup_{k>m} \frac{|F_k|}{|F_k\cap P|}\le\sup_{k>m}\frac{|F_k|}{\frac{|F_k|}{|E|+1}-1}=\sup_{k>m}\frac1{\frac1{|E|+1}-\frac1{|F_k|}}
\le \frac1{\frac1{|E|+1}-\frac1{2(|E|+1)}}=2(|E|+1)<\infty.$$
As the space $D$ is discrete and $C^*$-embedded into $X$, the bounded function $f$ has a continuous extension $\bar f:X\to\IR$. Since the space $X$ is Tychonoff, there exists a continuous function $\lambda:X\to[0,1]$ such that $\lambda(\bar D)=\{1\}$ and $\lambda(x)=0$ for all $x\in E\setminus \bar D$. Replacing $\bar f$ by the product $\bar f\cdot\lambda$, we can assume that $\bar f(x)=0$ for all $x\in E\cap\bar D$. We claim that $\bar f\in U$.

Given any $x\in E$ we should prove that $|\bar f(x)|<\e$. This is clear if $x\notin\bar D$. If $x\in F_k$ for some $k\le m$, then $|\bar f(x)|=|y_k|<\e$ as $y\in V$. If $x\in\bar D\setminus \bigcup_{k=1}^mF_k$, then $x\in \bar Q$ for some set $Q\in\mathcal P\setminus\{P\}$. The definition of the function $f$ ensures that $f|Q\equiv 0$ and then $|\bar f(x)|=0<\e$. This completes the proof of the inclusion $\bar f\in U$.

The definition of the function $f$ ensures that $\mu_k(\bar f)=\mu_k(f)=y_k$ for all $k\in\IN$. So, $(y_k)_{k=1}^\infty=T(f)\in T(U)$ and $V\subset T(U)$. This completes the proof of the openness of the operator $T:C_p(X)\to\ell_\infty\subset\IR^{\IN}$.  Since the kernel of the open operator $T$ equals to $$Z=\bigcap_{n=1}^\infty\{f\in C_p(X):\mu_n(f)=0\},$$
the quotient space $C_p(X)/Z$ is isomorphic to the subspace $T(X)=\ell_\infty$ of $\IR^{\IN}$.\hfill\qed

\section{Proof of Theorem \ref{main2}}

Let $X$ be a Tychonoff space and $(K_n)_{n\in\w}$ be a sequence of compact subsets of $X$ such that for every $n\in\w$ there are two embeddings $$\ddot e_{n,0},\ddot e_{n,1}:K_{n+1}\to K_n$$ such that $\ddot e_{n,0}(K_{n+1})\cap \ddot e_{n,1}(K_{n+1})=\emptyset$. Replacing the sequence $(K_n)_{n\in\w}$ by the sequence $(K_{2n})_{n\in\w}$, if necessary, we can assume that for any $n\in \w$ the set $K_n\setminus (\ddot e_{n,0}(K_{n+1})\cup \ddot e_{n,1}(K_{n+1}))$ is not empty and hence contains a point $\dot x_n$.

 Let $2^{<n}:=\bigcup_{k<n}2^k$ and $2^{\leq n}:=\bigcup_{k\leq n}2^k$ for $n\in \w$, where $2^k$ is the family of all binary sequences of length $k,$ that is $$2^0=\{\emptyset\},\,\, 2^1=\{(0), (1)\},\,\,\, 2^2=\{(0,0), (0,1), (1,0), (1,1)\}, \mbox{ \ \ and so on.} $$ For a binary sequence $s=(s_0,\dots,s_{n-1})\in 2^n$ by $|s|$ we denote the length $n$ of the sequence $s$; for $s=\emptyset \in 2^0$ we put $|s|=0.$

Let $s\hat{\;}i=(i)$ if $s=\emptyset \in 2^0, \,\, i\in \{0,1\},$ and $s\hat{\;}i=(s_0, \ldots, s_{n-1}, i)$\,\, if $s= (s_0, \ldots, s_{n-1}) \in 2^n,\,\, 1\leq n < \w, \,\, i\in \{0,1\}.$ Similarly we define $s\hat{\;}p\in 2^{|s|+|p|}$ for all $s,p \in 2^{< \w}.$

Consider the family of embeddings $$\big(e_s:K_{|s|}\to X\big)_{s\in 2^{<\w}}$$ defined by the recursive formula: $e_\emptyset:K_0\to X$ is the identity embedding of $K_0$ into $X$ and
$$e_{s\hat{\;}i}=e_s\circ \ddot e_{|s|,i}\mbox{ for $s\in 2^{<\w}$ and $i\in\{0,1\}$}.$$

For every $s\in 2^{<\w}$, let $K_s:=e_s(K_{|s|})$ and $x_s:=e_s(\dot x_{|s|})\in K_s$.

If $n\in \w, i\in\{0,1\}, s\in 2^n$ and $t=s\hat{\;}i \in 2^{n+1},$ then $K_t \subset K_s;$ indeed, $$K_t= K_{s\hat{\;}i}=e_{s\hat{\;}i} (K_{n+1})=e_s(\ddot e_{n,i}(K_{n+1}))\subset e_s(K_n)=K_s.$$

If $n\in \w$ and $s, t\in 2^n$ with $s\neq t$, then $K_s \cap K_t =\emptyset$.  Indeed, for $n=0$ it is obvious. Assume that it is true for some $n\in \w.$ Let $s, t\in 2^n$ and $ i, j \in \{0,1\}$ with $s\hat{\;}i\neq t\hat{\;}j.$ If $s\neq t$, then $$K_{s\hat{\;}i}\cap K_{t\hat{\;}j} \subset K_s \cap K_t =\emptyset.$$ If $s=t$, then $i\neq j$, so $K_{s\hat{\;}i}\cap K_{t\hat{\;}j}= K_{s\hat{\;}i}\cap K_{s\hat{\;}j} =e_s(\ddot e_{n,i}(K_{n+1}))\cap e_s(\ddot e_{n,j}(K_{n+1}))=\emptyset$ since $e_s$ is injective and $\ddot e_{n,i}(K_{n+1})\cap \ddot e_{n,j}(K_{n+1})=\emptyset.$ Thus for all $u,v \in 2^{n+1}$ with $u\neq v$ we have $K_u\cap K_v=\emptyset .$

If $n\in \w, i\in\{0,1\}$ and $t\in 2^n$, then $$x_t= e_t(\dot x_n) \in e_t (K_n \setminus \ddot e_{n,i}(K_{n+1}))= e_t(K_n) \setminus e_t(\ddot e_{n,i}(K_{n+1}))= K_t \setminus e_{t\hat{\;}i}(K_{n+1})= K_t \setminus K_{t\hat{\;}i}.$$ It follows that $x_t \not \in K_s$ if $t,s \in 2^{< \w}$ with $|t| < |s|.$

If $s,t,p \in 2^{< \w}$ with $|s|=|t|, $ then $e_s^{-1} \circ e_{s\hat{\;}p}=e_t^{-1} \circ e_{t\hat{\;}p}$ and $$e_s^{-1} (x_{s\hat{\;}p})= e_s^{-1} (e_{s\hat{\;}p}(\dot x_{|s|+|p|}))=e_t^{-1} (e_{t\hat{\;}p}(\dot x_{|t|+|p|}))= e_t^{-1} (x_{t\hat{\;}p}).$$

For every $n\in\w$ let $$\mu_n=2^{-n}\sum_{s\in 2^n}\delta_{x_s}$$ be the uniformly distributed probability measure with the finite support $\{x_s:s\in 2^n\}$. Let $$Z=\bigcap_{n\in\w}\{f\in C_p(X):\mu_n(f)=0\}.$$ The set $\{\mu_n: n\in \w\}$ is linearly independent in the dual of the space $C_p(X),$ since $x_s\neq x_t$ for $s,t\in 2^{<\w}$ with $s\neq t.$ Thus the quotient space $C_p(X)/Z$ is infinite-dimensional.
\smallskip

We  prove that this quotient space is metrizable.

Let $X_0= \{x_s: s\in 2^{<\w}\}$ and $K'_n=\bigcup_{s\in 2^n} K_s$ for $n\in \w.$ Clearly, $(K'_n)$ is a decreasing sequence of compact subsets of $K_0$ and $X_0 \setminus K'_n \subset \{x_s: s\in 2^{<n}\}$ for any $n\in \w$. The subset
$$K=X_0 \cup\bigcap_{n\in\w}K'_n$$
of $X$ is compact. Indeed, $\overline{ X_0}= \overline{ (X_0\setminus K'_n)} \cup \overline{ (X_0\cap K'_n)} \subset X_0 \cup K'_n$ for any $n\in \w,$ so $\overline{ X_0} \subset X_0 \cup \bigcap_{n\in \w} K'_n$. Thus $$\overline{ K} = \overline{ X_0} \cup \bigcap_{n\in \w} K'_n= X_0 \cup \bigcap_{n\in \w} K'_n = K\subset K_0,$$ so $K$ is compact.


Consider the quotient map $q:C_p(X)\to C_p(X)/Z$. For every $n\in\w$ consider the open neighborhood $$V_n:=\{f\in C_p(X):|f(x_s)|<2^{-n}\; \mbox{for each}\; s\in 2^{<n} \}$$ of zero in $C_p(X)$. We claim that $\big(q(V_n)\big)_{n\in\w}$ is a neighborhood basis at zero in the quotient space $C_p(X)/Z$.

Given any neighborhood $U\subset C_p(X)$ of zero, we need to find $n\in\w$ such that $V_n+Z\subset U+Z$. Without loss of generality we can assume that $U$ is of basic form $$U=\{f\in C_p(X):\max_{x\in E}|f(x)|<\e\}$$ for some $\e>0$ and some finite set $E\subset X$. Choose $n\in\w$ so large that
\begin{enumerate}
\item $2^{-n}<\e$;
\item $2^{1-n}|E|<1$;
\item $E\cap X_0 \subset \{x_s\}_{s\in 2^{<n}}$;
\end{enumerate}

We claim that $V_n\subset U+Z$. Given any function $f\in V_n$, we should write it as $f=u+\zeta$ where $u\in U$ and $\zeta\in Z$.

Consider the set $S=\{s\in 2^n:E\cap K_s\ne\emptyset\}$. The condition (2) ensures that $|S|\leq |E|\leq 2^{n-1}$. So, we can find an injective map $\xi:S\to 2^n\setminus S$.  Now define the function $\zeta_0: K\to \mathbb R$ by the formula
$$\zeta_0(x)=\begin{cases}
f(x)&\mbox{$x\in K_s\cap K$ for some $s\in S$};\\
-f\circ e_s\circ e^{-1}_{\xi(s)}(x)&\mbox{if $x\in K_{\xi(s)}\cap K$ for some $s\in S$};\\
0&\mbox{otherwise}.
\end{cases}
$$

The function  $\zeta_0$ is continuous. Indeed, let $B_s=\{x_s\}$ for $s\in 2^{<n}$ and $B_s=K_s\cap K$ for $s\in 2^n$. Then $K=\bigcup_{s\in 2^{\leq n}} B_s$ and the sets $B_s, s\in 2^{\leq n},$ are compact and pairwise disjoint, so they are open and closed subsets of $K$. Since $\zeta_0|B_s$ is continuous for any $s\in 2^{\leq n}$, then $\zeta_0$ is continuous. As $K$ is compact and $E$ is finite,  there exists $\zeta \in C_p(X)$ with $\zeta|K=\zeta_0$ and $\zeta (x) = f(x)$ for all $x\in E\setminus K.$

We  prove that $\zeta \in Z.$

Let $m\in \w.$ If $m<n$ and $t\in 2^m,$ then $\zeta (x_t)=0,$ so $\mu_m (\zeta)=0$ for $m<n.$

If $m\geq n, $ then $$
\begin{aligned}
 \sum_{t\in 2^m} \zeta(x_t)&= \sum_{p\in 2^{m-n}} \sum_{s\in 2^n} \zeta(x_{s\hat{\;}p})=\sum_{p\in 2^{m-n}} \Big(\sum_{s\in S}\zeta(x_{s\hat{\;}p}) + \sum_{s\in S}\zeta(x_{\xi(s)\hat{\;}p})\Big)=\\
&=\sum_{p\in 2^{m-n}} \Big(\sum_{s\in S}f(x_{s\hat{\;}p}) - \sum_{s\in S}f\big(e_s(e_{\xi(s)}^{-1}(x_{\xi(s)\hat{\;}p}))\big)\Big)=\\
&=\sum_{p\in 2^{m-n}} \Big(\sum_{s\in S}f(x_{s\hat{\;}p}) - \sum_{s\in S}f\big(e_s(e_s^{-1}(x_{s\hat{\;}p}))\big)\Big)=\\
&=\sum_{p\in 2^{m-n}}\Big(\sum_{s\in S}f(x_{s\hat{\;}p})- \sum_{s\in S}f(x_{s\hat{\;}p})\Big)=0,
\end{aligned}
$$ so $\mu_m (\zeta)=0$ for $m\geq n.$
Thus $\zeta \in Z.$

Finally we  prove that $f-\zeta \in U.$ For $x\in E\setminus K$ we have $|f(x)-\zeta (x)|=0<\e.$ Let $x\in E\cap K.$ Then $x=x_t$ for some $t\in 2^{<n}$ or $x\in K_s$ for some $s\in S.$ In the first case we have $$|f(x)-\zeta(x)|=|f(x)-0|<2^{-n}<\e;$$ in the second case we get $$|f(x)-\zeta(x)|=|f(x)-f(x)|=0<\e.$$ Thus $|f(x)-\zeta(x)|<\e$ for any $x\in E,$ so $f-\zeta \in U.$

We have shown that the quotient space $C_p(X)/Z$ is infinite-dimensional and metrizable.

\end{document}